\documentclass[12pt,a4paper]{article}
\usepackage{graphicx}
\usepackage{amsfonts}
\usepackage{amssymb}
\usepackage{amsthm}
\usepackage[centertags]{amsmath}
\usepackage{newlfont}

\usepackage{color}
\usepackage{graphicx,color}
\usepackage{amsmath, amssymb, graphics}

\usepackage{graphicx}
\usepackage{amsfonts}
\usepackage{amssymb}
\usepackage{amsmath}
 \def\r{\mathbb{R}}
 \def\l{\mathbb{L}}
 \def\c{\mathbb{C}}
  \def\e{\mathbb{E}}
    \def\h{\mathbb{H}}
        \def\RE{\mathrm{Re}}
                \def\IM{\mathrm{Im}}

\newtheorem{theorem}{Theorem}[section]

\newtheorem{proposition}[theorem]{Proposition}

\newtheorem{lemma}[theorem]{Lemma}

\title{The Lorentzian version of a theorem of Krust}
\author{Rafael L\'opez\footnote{Partially supported by the grant no. MTM2017-89677-P, MINECO/AEI/FEDER, UE}\\
 Departamento de Geometr\'{\i}a y Topolog\'{\i}a \\
 Universidad de Granada\\
 18071 Granada, Spain\\
\texttt{rcamino@ugr.es}}

\date{}
\begin{document}
\maketitle

\begin{abstract}  
In Lorentz-Minkowski space, we prove that the conjugate   surface of a maximal graph over a convex domain is also a graph.  We provide three proofs of this result that show a suitable correspondence between maximal surfaces in Lorentz-Minkowski space  and minimal surfaces in Euclidean space.  
\end{abstract}


\section{Introduction and statement of the result}

Romain Krust proved the following result about minimal graphs in Euclidean space $\e^3$: see \cite[p. 188]{dhkw} and \cite[p. 33]{ka}.

\begin{theorem}\label{t1}
 If an embedded minimal surface $X:B\rightarrow\e^3$, $B=\{w\in\c:|w|<1\}$, can be written as a graph over a convex domain in a plane, then the conjugate surface $X^*:B\rightarrow\e^3$ is a graph. 
\end{theorem}

In this paper, we extend this result for maximal surfaces in Lorentz-Minkowski   space $\l^3$.  An   immersion of a surface in $\l^3$ is called spacelike if the induced metric on the surface is   Riemannian.  A maximal surface in $\l^3$ is a   spacelike surface   with  zero mean curvature  at every point. From the variational viewpoint, maximal surfaces locally  represent a maximum for the area integral. The Lorentzian version of the Krust's theorem   is formally the same, except that we need to precise the causal character of the plane with respect to which the surface is a graph.

\begin{theorem} \label{t2}
If an embedded maximal surface $X:B\rightarrow\l^3$, $B=\{w\in\c:|w|<1\}$, can be written as a graph over a convex domain in a spacelike plane, then the conjugate surface $X^*:B\rightarrow\l^3$ is a graph. 
\end{theorem}

Although this result may be expected, there are differences in the theory of minimal surfaces in $\e^3$ and maximal surfaces in $\l^3$. A  clear example   is the Bernstein theorem. In $\l^3$, the only maximal entire graphs  are  spacelike planes (\cite{cy}), which is the Lorentzian version of the classical Bernstein theorem in Euclidean space. However, the result holds for arbitrary dimensions, that is, for maximal hypersurfaces in $\l^{n}$, in contrast with the Euclidean version, where the Bernstein theorem  only is valid for minimal hypersurfaces of $\e^n$ with $n\leq 7$ (\cite{bo}). 

The goal of this paper is to provide three different approaches of the proof of Theorem \ref{t2}. A first proof follows the same steps as in Euclidean space by means of the Weierstrass representation of a maximal surface (Section \ref{sec3}). The second proof makes use of a clever geometric  idea due to F. Mart\'{\i}n in \cite{ma} (Section \ref{sec4}). Finally, the third proof uses a duality correspondence between maximal surfaces of $\l^3$ and minimal surface of $\e^3$  (Section \ref{sec5}). 

\section{Preliminaries}\label{sec2}

In this section we fix the terminology and notation. The Lorentz-Minkowski $3$-dimensional space $\l^3$ is   the vector space $\r^3$ with canonical coordinates  $(x_1,x_2,x_3)$ and endowed with the metric $\langle,\rangle=dx_1^2+dx_2^2-dx_3^2$. A vector $v\in\r^3$ is spacelike, timelike or lightlike if $\langle v,v\rangle$ is positive, negative or zero, respectively.  We refer  the reader to \cite{lo5} for general definitions of $\l^3$. In order to distinguish   the Euclidean space from the Lorentzian space, we denote the Euclidean space by $\e^3$, that is,  $\r^3$ with the Euclidean metric $\langle,\rangle_0=dx_1^2+dx_2^2+dx_3^2$. 

The spacelike condition for a surface in $\l^3$ is a strong property. For example, any spacelike (connected) surface $X:M\rightarrow\l^3$ is orientable. This is due to the fact the two timelike vectors are not orthogonal. Indeed, the orthogonal subspace to each tangent plane $T_pM$ is timelike. Thus, if $N$ is a local orientation on $M$, $N(p)$ is a unit timelike vector and if $e_3=(0,0,1)$,   then $|\langle N(p),e_3\rangle|\geq 1$ for every $p\in M$. Hence,   by connectedness, it is possible to choose a global orientation $N$ on $M$ such that $\langle N,e_3\rangle\leq -1$ globally in $M$ (or $\langle N,e_3\rangle\geq 1$ in $M$), which proves the orientability of $M$. In this paper, we will choose the orientation on a spacelike surface such that  $\langle N,e_3\rangle\leq -1$ on $M$.

Other consequence of the spacelike condition is when we consider spacelike graphs in $\l^3$.   It is known that any surface $X:M\rightarrow \r^3$ (without any induced metric) is locally the graph over one of the three coordinate planes of $\r^3$. However, if $X:M\rightarrow \l^3$  is a  spacelike immersion,   then  we can assure that the surface is  locally a graph on the $x_1x_2$-plane. Indeed, consider the orthogonal projection onto the $x_1x_2$-plane, which we identify with $\r^2=\c$: 
$$\pi:\r^3\rightarrow\r^2,\quad  \pi(x_1,x_2,x_3)=(x_1,x_2).$$ Define the map   $\tilde{x}=\pi\circ X:M \rightarrow\r^2$. The differential $(d\tilde{x})_p$ at $p$ is $(d\tilde{x})_p(v)=(v_1,v_2)$, $v=(v_1,v_2,v_3)\in T_pM$. Then  
\begin{equation}\label{pr}
|(d\tilde{x})_p(v)|^2= v_1^2+v_2^2\geq v_1^2+v_2^2-v_3^2 =\langle v,v\rangle>0
\end{equation}
for any nonzero tangent vector $v$, and consequently,   $(d\tilde{x})_p$ is injective. This   proves that $\tilde{x}$ is a local diffeomorphism. 

In the classical theory of minimal surfaces in Euclidean space, it is an issue to determine when a minimal surface is a minimal graph. One of the first results was obtained by   Rad\'o proving that  if $X:B\rightarrow\e^3$  is a compact minimal disk and $X(\partial B)$   can be  orthogonally projected one-to-one onto a planar convex closed curve $\Gamma$, then $X(B)$ is a minimal graph over the convex planar domain bounded by $\Gamma$ (\cite{ra}). For spacelike surfaces, without assuming any assumption on its mean curvature, the result   goes beyond.

\begin{proposition} 
Let $\Gamma\subset\r^3$ be a simple closed curve and let $X:M \rightarrow\l^3$ be a compact spacelike surface such that $X:\partial M\rightarrow \Gamma$ is a diffeomorphism. If   there exists a spacelike plane $P$ such that the orthogonal projection   of $\Gamma$ on $P$ is a simple closed curve, then $X(M)$ is  a spacelike graph on some domain of $P$.  In particular, a compact spacelike surface      spanning a   simple closed planar curve is a graph.
\end{proposition}

\begin{proof}
After a rigid motion, we can assume  that $P$ is the   plane of equation $ x_3=0$. By \eqref{pr}, we know that  
$\tilde{x}=\pi\circ X:\mbox{int}(M )\rightarrow\r^2$ is a local diffeomorphism and, therefore, it is an open map. Let
$\Omega=\tilde{x}(\mbox{int}(M ))\subset\r^2$, which is an open subset in $\r^2$,
and let $\Omega'$ be the planar domain bounded by the plane simple closed curve
$\Gamma'=\pi(\Gamma)=\tilde{x}(\partial M )$.

\begin{enumerate}
\item Claim: $\partial \tilde{x}(M )\subset\Gamma'$.

Since $M $ is compact, for any $q\in\partial \tilde{x}(M )$ there exists
$p\in M $ such that $\tilde{x}(p)=q$. We show that $p\in\partial M $. On the contrary,
 $p\in\mbox{int}(M )$ and there is an open neighborhood $U_p$ of $p$ in $\mbox{int}(M )$ and
an open neighborhood $V_q$ of $q$ in $\Omega$ such that $\tilde{x}:U_{p}\rightarrow V_{q}$
is a diffeomorphism. This implies that $q\in\Omega$, contradicting that $q$ is a boundary point of $\tilde{x}(M )$.
\item Claim:  $\Omega=\Omega'$. 

If there exists a point in $\Omega$ which is not in $\Omega'$ and since $\Omega$ is bounded, there are points in $\partial\Omega$ outside $\Omega'$, which is impossible.
Analogously, if there is a point in $\Omega'$ which is not in $\Omega$, there are points in $\partial\Omega$ inside $\Omega'$, which is not possible again.
\end{enumerate}
  As a consequence, $\tilde{x}:M \rightarrow\overline{\Omega}$ is a local diffeomorphism, and the compactness of $M $ implies that $\tilde{x}$ is a covering map. Since $\overline{\Omega}$ is simply connected, the map $\tilde{x}$ must be a global diffeomorphism.

Hence, letting
$F=\tilde{x}^{-1}$ we conclude that $x\circ F$ is the graph determined by the function $f=x_3\circ F$.

For the last statement, assume that the boundary $\Gamma$ of the surface is planar. Because any curve contained in a spacelike surface is spacelike, the curve $\Gamma$ is spacelike. Since   $\Gamma$ is a closed curve,   the plane   containing $\Gamma$ must be spacelike (\cite{lo5}). This proves that $\Gamma$ is contained in a spacelike plane, and the result applies. 
\qed\end{proof}

Minimal surfaces in $\e^3$ and maximal surfaces in $\l^3$ share some properties.  For example, maximal surfaces admit   a Weierstrass representation as it occurs for minimal surfaces and  that we now explain (\cite{ko}). Let $M$ be an orientable surface and consider isothermal parameters on $M$ which induce a conformal structure on $M$. Let $X:M\rightarrow\l^3$ be a spacelike conformal immersion such that the mean curvature vanishes at every point of $M$, that is, $X$ is a maximal surface.     If  $N$ is the Gauss map of $X$, and because $\langle N,e_3\rangle\leq -1$,  then $N$ is a map  
$$N:M\rightarrow\h^2:=\{x=(x_1,x_2,x_3)\in\r^3: \langle x,x\rangle=-1,x_3\geq 1\}.$$
 If $\overline{\c}=\c\cup\{\infty\}$ is the extended complex plane and $\mu:\h^2\rightarrow \overline{\c}\setminus\{|z|=1\}$  is the  stereographic projection from the North pole $(0,0,1)\in\h^2$, the Gauss map $N$ is viewed as a meromorphic function $g:M\rightarrow\overline{\c}$, with  $g=\mu \circ N$.

Define a $\c^3$-valued holomorphic $1$-form on $M$ by $\Psi=2\,dX=2\,X_z\, dz$, where $z$ is a complex coordinate of $M$. 
There is a holomorphic $1$-form $\eta$ on $M$  such that the $1$-forms
$$\Psi_1=\frac{1}{2}(1+g^2)\eta,\quad\Psi_2=\frac{i}{2}(1-g^2)\eta,\quad\Psi_3=-g\eta$$
are holomorphic on $M$ without common zeroes and they   have no real periods. If $z_0\in M$ is a base point, then  the immersion $X$ is determined by  
\begin{equation}\label{wr}
X(z)=X(z_0)+\RE\int_{z_0}^z\Psi,\quad \Psi=(\Psi_1,\Psi_2,\Psi_3).
\end{equation}
We say that  $(M,g,\eta)$ (or $(M,\Psi)$) are the Weierstrass data of $X$ and \eqref{wr} is the Weierstrass representation of $X$. The complex curve $\Psi$ associate to $X$ is isotropic in the sense that $\langle\Psi,\Psi\rangle=\Psi_1^2+\Psi_2^2-\Psi_3^2=0$, where $\langle,\rangle$  denotes the complexification   of the Lorentzian metric. 

Suppose that $X$ is defined on a simply connected  domain $\Omega$ of $\r^2=\c$. The {\it conjugate surface} $X^*:\Omega\rightarrow\l^3$ is defined on $\Omega$ as solution of the Cauchy-Riemann equations
$$X_u^*=-X_v,\quad X_v^*=X_u$$
in $\Omega$, where $z=u+iv\in\Omega$, $i=\sqrt{-1}$. Then $X^*:\Omega\rightarrow\l^3$ is a conformal spacelike immersion which   is also a maximal surface.  The map $X^*$ is nothing that the harmonic conjugate of $X$, that is, the map $\tilde{X}=X+iX^*:\Omega\rightarrow\c^3$ is   holomorphic   and the complex derivative of $\tilde{X}$ is $\tilde{X}_z=X_u+iX_u^*=X_u-iX_v=\Psi/2$. In particular, 
up to a constant, $X^*(z)=\IM\int^z\Psi dz$.  Moreover 
\begin{equation}\label{dx}
dX^*=-dX\circ J,
\end{equation}
 where $J$ is the rotation of 90 degrees on all tangent planes induced by the orientation on $\Omega$. More precisely, if $\{\partial_u,\partial_v\}$ is the oriented basis in the tangent plane determined by the conformal parameter $z$, then 
 $$J(\partial_u)=\partial_v,\quad J(\partial_v)=-\partial_u.$$
The conjugate surface $X^*$ is isometric to the initial surface $X$ and both surfaces  have the same Gauss map at corresponding points. The isotropic curve of $X^*$ is
\begin{equation}\label{pp}
\Psi^*=2\,dX^*=2\,X^*_z dz=2(X^*_u-iX^*_v)dz=-2i(X_u-iX_v)dz=-i\,\Psi,
\end{equation}
and the Weierstrass data of $X^*$ is $(\Omega,g,-i\,\eta)$.

\section{First proof: using the Weierstrass representation}\label{sec3}
 
 In this section we follow the same steps that the original idea of Krust by means of the  Weierstrass representation formula for a maximal surface. Firstly, we change the expression \eqref{wr} by defining a meromorphic function $h$ such that $dh=g\eta$. Then \eqref{wr} is now
 \begin{equation}\label{wr2}
 X(w)=X(w_0)+\RE\int_{w_0}^w \left(\frac12\left(\frac{1}{g}+g\right),\frac{i}{2}\left(\frac{1}{g}-g\right),-1\right)dh.
 \end{equation}
 After a rigid motion of $\l^3$, we can assume   that $X$   is graph on a convex domain $D$ of the plane of equation $x_3=0$, which we identify with $\r^2=\c$.  By our choice of orientation for a spacelike surface explained in Section \ref{sec2},   the Gauss map $N$ points upwards in $B$.  Because the inverse of the stereographic projection is 
 $$\mu^{-1}(z)=\left(\frac{-2\,\RE z}{|z|^2-1},\frac{-2\,\IM z}{|z|^2-1},\frac{|z|^2+1}{|z|^2-1}\right)$$
 and $N=\mu^{-1}\circ g$, we deduce that   $|g(z)|>1$ in $B$. 
 
 Without loss of generality, we assume  $w_0=0$ and $X(w_0)=0$ in \eqref{wr2}.  Let us introduce the following notation:
 $$\sigma(w)=-\int_0^w\frac{g}{2}dh,\quad \tau(w)=\int_0^w\frac{1}{2g}dh.$$
 Using $\sigma$ and $\tau$ together \eqref{wr2}, the orthogonal projection of $X(B)$   is 
 \begin{eqnarray}\label{p1}
 \pi\circ X(w)&=&\RE\int_{0}^w \left(\frac12\left(\frac{1}{g}+g\right),\frac{i}{2}\left(\frac{1}{g}-g\right)\right)dh\nonumber\\
 &=&\RE\left(\tau-\sigma+i(t\tau+\sigma)\right)=\bar{\tau}-\sigma.
 \end{eqnarray}
 Analogously, and because $X^*(w)=\IM\int_0^{w}\Psi$, we have
 \begin{eqnarray}\label{p11}
 \pi\circ X^*(w)&=&\IM\int_{0}^w \left(\frac12\left(\frac{1}{g}+g\right),\frac{i}{2}\left(\frac{1}{g}-g\right)\right)dh\nonumber\\
 &=&\IM\left(\tau-\sigma+i(\tau+\sigma)\right)=i(\bar{\tau}+\sigma).
 \end{eqnarray}

 {\it Claim:} If $w_1\not=w_2$, $w_1, w_2\in B$, then $\pi\circ X^*(w_1)\not=\pi\circ X^*(w_2)$.
 
 Let $p_i=\pi\circ X(w_i)$ and $q_i=\pi\circ X^*(w_i)$, $i=1,2$. Since $X(B)$ is a graph on the convex domain $D=\pi\circ X(B)$, there is a segment in $D$ connecting $p_1$ with $p_2$. We parametrize this segment by 
 $$\gamma(t):[0,1]\rightarrow D,\quad \gamma(t)=(1-t)p_1+t p_2$$
  and let 
 $\beta:[0,1]\rightarrow B$ be a   curve in the unit disk $B$ such that $X\circ\beta=\pi_{|X(B)}^{-1}\circ  \gamma$.  Then
 $$p_2-p_1=\gamma(1)-\gamma(0)=\gamma'(t).$$
 Using \eqref{p1},
 \begin{eqnarray}\label{s1}
 p_2-p_1&=&d(\pi\circ X)\beta'(t)=d(\bar{\tau}-\sigma)\beta'(t)\nonumber\\
 &=&\left(\overline{\frac{h'(w)}{2g(w)}}-\frac{g(w)h'(w)}{2}\right){\Bigg|}_{w=\beta(t)}\beta'(t)
 \end{eqnarray}
Similarly, taking into account \eqref{p11},  we obtain
 \begin{eqnarray}\label{s11}
 q_2-q_1&=&\int_\gamma d(\pi\circ X^*)\beta'(t)=\int_\gamma i\, d(\bar{\tau}+\sigma)\beta'(t)\nonumber\\
 &=& \int_\gamma i \left(\overline{\frac{h'(w)}{2g(w)}}+ \frac{g(w)h'(w)}{2}\right){{\Bigg |}_{w=\beta(t)}}\beta'(t)
 \end{eqnarray}
  We multiply  $p_2-p_1$ and $i(q_2-q_1)$ with the Euclidean scalar product $\langle,\rangle_0$ of $\r^2$. Recall that in complex notation,  $\langle v_1,v_2\rangle_0=\RE(v_1\overline{v_2})$,   $v_1,v_2\in\r^2$.  Thus \eqref{s1} and \eqref{s11} imply
\begin{eqnarray*}
\langle p_2-p_1,i(q_2-q_1)\rangle&=&\RE((p_2-p_1)\overline{(i(q_2-q_1))})\\
&=&-\int_0^1\RE\left(\overline{\frac{h'(w)}{2g(w)}}-\frac{g(w)h'(w)}{2}\right)\left( \frac{h'(w)}{2g(w)}+\overline{\frac{g(w)h'(w)}{2}}\right){\Bigg|}_{w=\beta(t)}|\beta'(t)|^2\\
&=&-\int_0^1  \frac{|\beta'(t)|^2}{4}\left( \frac{1}{|g|^2}-|g|^2\right){{\Bigg |}_{w=\beta(t)}}>0.
\end{eqnarray*}
 Because $p_1\not=p_2$, we conclude $q_1\not=q_2$. This proves that the orthogonal projection $\pi:X^*(B)\rightarrow\r^2$  is injective, hence $X^*(B)$ is a graph.

 \section{Second proof: a geometric approach}\label{sec4}
 
In this proof, we begin as in the above section and we employ the notation that appeared there. Without loss of generality, we assume that $X(B)$ is a graph on the $x_1x_2$-plane of $\r^3$, which we identify with $\r^2$ again.    

{\it Claim:} If $w_1\not=w_2$, $w_1, w_2\in B$, then $\pi\circ X^*(w_1)\not=\pi\circ X^*(w_2)$.
 
Recall that $\gamma$ is the segment in $D$ that connects $p_1$ with $p_2$.  Let $\Pi$ be the plane containing  $\gamma$ and orthogonal to $\r^2$. Let $\vec{a}=(a_1,a_2,0)\in\l^3$ be a unit vector orthogonal to $\Pi$. Define the curve $\alpha:[0,1]\rightarrow \l^3$ by $\alpha(t)=X(\beta(t))$, which connects $X(w_1)$ with $X(w_2)$. Denote $\alpha^*=X^*\circ\beta$ the conjugate curve of $\alpha$ connecting $X^*(w_1)$ with $X^*(w_2)$. 
 
If we write $\beta'(t)=u(t)\partial_u(t)+v(t)\partial_v(t)$, where $\{\partial_u,\partial_v\}$ is an oriented basis at each tangent plane of the $z$-plane, by \eqref{dx}, we obtain 
\begin{eqnarray*}
{\alpha^*}{'}(t)&=&(dX^*)(\beta'(t))=-(dX)(-v(t)\partial_u(t)+u(t)\partial_v(t)).\\
&=&v(t)X_u(t)-u(t)X_v(t)
\end{eqnarray*}

 On the other hand, if  $N=\lambda (X_u\times X_v)$ is the Gauss map of $X$, $\lambda>0$,  where $\times$ is the Lorentzian vector product in $\l^3$, and because $X$ is conformal, we have $N\times X_u=-X_v$ and $N\times X_v=X_u$. then
$$N(\alpha(t))\times \alpha'(t)=-u(t)X_v(t)+v(t)X_u(t)={\alpha^*}{'}(t).$$
Since $\alpha$ is   contained in the plane $\Pi$, the vector $\nu(t)=N(\alpha(t))\times \alpha'(t)$ does not belong to the plane $\Pi$, hence $\nu(t)$ satisfies 
$\langle\nu(t),\vec{a}\rangle\not=0$ for every $t\in [0,1]$.  Since $\vec{a}$ is a horizontal vector, $\langle\nu(t),\vec{a}\rangle=\langle\nu(t),\vec{a}\rangle_0$. Without loss of generality, we assume  $\langle\nu(t),\vec{a}\rangle>0$ in $[0,1]$. In particular, $\langle\pi\circ\nu(t),\vec{a}\rangle>0$ in $[0,1]$ because $\nu(t)\not\in\Pi$ for every $t\in [0,1]$. Then 
\begin{eqnarray*}\langle q_2-q_1,\vec{a}\rangle&=&\langle\int_0^1 {(\pi\circ\alpha^*)}{'}(t),\vec{a}\rangle\, dt=
 \langle\int_0^1\pi\circ\nu(t),\vec{a}\rangle\, dt\\
&=&\int_0^1\langle\pi\circ\nu(t),\vec{a}\rangle\, dt>0.
\end{eqnarray*}
Again, we conclude $q_1\not=q_2$. Thus   $\pi:X^*(B)\rightarrow\r^2$   is injective, hence $X^*(B)$ is a graph.
 
\section{Third proof: using duality}\label{sec5}

Between minimal surfaces and maximal surfaces there is a correspondence, called duality, that assigns to each  minimal surface in $\e^3$   a maximal  surface in $\l^3$  and {\it vice-versa} (see \cite{le2,lm} for generalizations  in other  ambient spaces). It was Calabi the first who realized of this correspondence  when the surfaces are expressed as graphs on simply connected domains  (\cite{ca}). Indeed, assume that  $S$ is   a minimal graph in $\e^3$ of a function $f:\Omega\subset\r^2\rightarrow\r$ defined in a simply connected domain $\Omega$. Then the minimality of $S$ is equivalent to 
$$\mbox{div}\left(\frac{Df}{\sqrt{1+|Df|^2}}\right)=0.$$
Since $\Omega$ is simply connected, there exists a solution $f^\flat:\Omega \rightarrow\r$ of the equation
$$Df^\flat=\frac{(-f_y,f_x)}{\sqrt{1+|Df|^2}}.$$
Moreover, $f^\flat$ satisfies  $|Df^\flat|<1$  in $\Omega$, that is, the graph $S^\flat$ of $f^\flat$, viewed in the Lorentz-Minkowski space $\l^3$, is spacelike. On the other hand, it is immediate that 
$$\mbox{div}\left(\frac{Df^\flat}{\sqrt{1-|Df^\flat|^2}}\right)=0,$$
which is equivalent to say  that the mean curvature of $S^\flat$ vanishes identically. Thus $S^\flat$   is a maximal surface which   is called the  {\it dual surface} of $S$. We say that this duality is by graphs because the arguments are local. It is immediate a reverse process starting from a maximal graph $S$ in $\l^3$   and  obtaining a minimal surface   $S^\sharp$ in $\e^3$, called the {\it dual surface} of $S$.

On the other hand,   a similar correspondence appeared in \cite{gu,lls}, where  the duality is    defined in terms of  the isotropic curve that determines the surface. Exactly, let $X:\Omega\subset\c \rightarrow\e^3$, $X=X(z)$, be a conformal minimal surface defined on a simply connected domain $\Omega$ and let   $\Phi:\Omega\rightarrow\c^3$ be the isotropic curve such that  
 $\Phi(z)=2X_zdz=(\Phi_1,\Phi_2,\Phi_3)$. Then  $\langle\Phi,\Phi\rangle_0=0$,   where $\langle,\rangle_0$ is the complexification in $\c^3$ of the Euclidean metric.  We consider the $1$-form in $\c^3$ defined by $\Psi=(\Phi_1,\Phi_2,-i\Phi_3)$. It is immediate that   $\langle\Psi,\Psi\rangle=0$. The {\it dual surface} of $X$ is   the maximal surface   $X^\flat:\Omega\rightarrow\l^3$  defined by $X^\flat(z)=\RE\int^z\Psi(z)$. The converse process is similar. If $X:\Omega \rightarrow\l^3$ is a conformal maximal surface and $\Psi=(\Psi_1,\Psi_2,\Psi_3)$ is its isotropic curve, then  $\Phi= (\Psi_1,\Psi_2,i\Psi_3)$  is the isotropic curve of a minimal surface       in  $\e^3$ by means of $X^\sharp(z)=\RE\int^z \Phi(z)$, which is called the {\it dual surface} of $X$.  Furthermore, and up to   translations of the ambient space,  we have $M=(M^\sharp)^\flat$.  If $\mbox{Min}$ and $\mbox{Max}$ denote the family of minimal surfaces of $\e^3$ and the  maximal surfaces of $\l^3$, respectively, the duality is given by the two maps
$$\flat:\mbox{Min}\rightarrow\mbox{Max},\quad\quad\sharp:\mbox{Max}\rightarrow\mbox{Min}$$ 
with the property that $\flat\circ\sharp$ and $\sharp \circ\flat$ are the identities in $\mbox{Max}$ and $\mbox{Min}$ respectively.

It is important to remark that   both definitions of duality coincide. This was proved by  Lee    and the   key of this equivalence is the existence of a simultaneous conformal coordinates for a minimal graph and   its dual maximal graph.

\begin{proposition}[\cite{le}]\label{pr2}
Up to a translation, the dual surface of a minimal (resp. maximal) graph over a simply connected domain of $\r^2$ coincides with the dual surface obtained by the correspondence $\flat$ (resp. $\sharp$).
\end{proposition}

Other ingredient that we need in the proof of Theorem \ref{t2} is   the behavior of duality by conjugations of  minimal surfaces and maximal surfaces.

\begin{lemma}\label{l1}
 Up to translations of the ambient space, the duality   and the conjugation processes are commutative, that is, $(M^\flat)^*=(M^*)^\flat$ and $(M^\sharp)^*=(M^*)^\sharp$.
\end{lemma}
\begin{proof} We only prove the first identity because the other one is analogous. Let $X:M\rightarrow\e^3$ be a minimal surface and let $\Phi=(\Phi_1,\Phi_2,\Phi_3)$ be its isotropic curve. The isotropic curve of $X^\flat$ is $(\Phi_1,\Phi_2,-i\Phi_3)$. By \eqref{pp},  the isotropic curve of  $(X^\flat)^*$ is $-i(\Phi_1,\Phi_2,-i\Phi_3)=(-i\Phi_1,-i\Phi_2,-\Phi_3)$.

On the other hand, we know that the isotropic curve $\Phi^*$ of   $X^*:M\rightarrow\e^3$ is 
$\Phi^*=-i\Phi=- i(\Phi_1,\Phi_2,\Phi_3)=(-i\Phi_1,-i\Phi_2,-i\Phi_3)$. Then the isotropic curve of $(X^*)^\flat$ is  $(-i\Phi_1,-i\Phi_2,-\Phi_3)$, which coincides with the one of $(X^\flat)^*$, proving the result.\qed
\end{proof}

We now present the third proof of Theorem \ref{t2}. By using the dual correspondence, we carry the proof   in a problem in Euclidean space $\e^3$. Then we  use   Theorem \ref{t1}, and finally we come back to $\l^3$ by duality again. A similar idea was used by   Al\'{\i}as and Palmer   to prove the equivalence of the Bernstein theorem in $\e^3$ and  $\l^3$ (\cite{ap}). 

Let $X:B\rightarrow\l^3$ be a maximal graph on a convex domain $\Omega$ of a spacelike plane $P$. After a rigid motion, we assume   that $P$ is the $x_1x_2$-plane, which we identify   with $\r^2$. Let $X^\sharp:B\rightarrow\e^3$ be its dual surface (as a minimal graph). Since   $\Omega$ is simply connected, we know by Proposition \ref{pr2} that $X^\sharp(B)$  is a minimal graph on the same domain $\Omega$. Because $\Omega$ is convex,  the (Euclidean) Krust's Theorem \ref{t1} asserts that the dual surface $(X^\sharp)^*:B\rightarrow\e^3$ is a minimal graph on some domain $\widetilde{\Omega}\subset\r^2$. Because $(X^\sharp)^*$ is an embedding and $B$ is the unit ball, the domain $\widetilde{\Omega}$ is simply connected. Using Proposition \ref{pr2} again, the dual surface of  $(X^\sharp)^*$, namely,   $((X^\sharp)^*)^\flat$, is a maximal graph on the same domain $\widetilde{\Omega}$. Finally, by Lemma \ref{l1}, 
$$((X^\sharp)^*)^\flat=((X^\sharp)^\flat)^*=X^*,$$
proving that $X^*$ is a maximal graph. This  concludes the proof.



\end{document}